\documentclass[reqno,preprint,12pt]{elsarticle}
\usepackage[active]{srcltx}
\usepackage{bbm}
\usepackage{tipa}
\usepackage{amsthm}
\usepackage{amssymb}
\usepackage{stackrel}
\usepackage{mathrsfs}
\usepackage{amsmath,amscd}
\usepackage{lscape}
\usepackage{longtable}
\usepackage{booktabs}
\usepackage{rotating}
\usepackage[active]{srcltx}
\usepackage{amsfonts}
\usepackage[centerlast]{caption}
\usepackage{latexsym}
\usepackage{amsmath,amsthm}
\usepackage{amssymb,amsmath}
\usepackage[all]{xy}
\usepackage{graphicx}
\usepackage{rotating}
\usepackage[dvips]{geometry}
\newtheorem{dl}{Theorem}[section]
\newtheorem{tl}[dl]{Corollary}
\newtheorem{yl}[dl]{Lemma}
\newtheorem{dy}[dl]{Definition}
\newtheorem{lz}[dl]{Example}

\newtheorem{xinzhi}[dl]{Proposition}

\newtheorem{remark}[dl]{Remark}

\numberwithin{equation}{section}

\newproof{pot1}{Proof of Theorem \ref{ma1}}
\newproof{pot2}{Proof of Theorem \ref{ma2}}
\newtheorem{wenti}{\bf Research problem}[section]
\def\qed{\hfill \rule{4pt}{7pt}}
\def\pf{\noindent {\sl Proof.~}}

\usepackage[colorlinks=true,linkcolor=green,filecolor=brown,citecolor=green]{hyperref}
\usepackage{hyperref}
\usepackage{epsf}

\begin{document}
\title{Some nonlinear  inverse relations of the Bell polynomials via the Lagrange inversion formula}
\author{Jin Wang\fnref{a,fn2}}
\fntext[fn2]{E-mail address: jinwang@zjnu.edu.cn}
\address[a]{College of Mathematics and Computer Science, Zhejiang Normal University, Jinhua 321004, P.~R. ~China}
\author{Xinrong Ma\fnref{b,fn3}}
\fntext[fn3]{E-mail address: xrma@suda.edu.cn}
\address[b]{Department of Mathematics, Soochow University, Suzhou 215006, P.~R. ~China}
\begin{abstract}
In this paper, by means of the classical Lagrange inversion formula, we establish a general  nonlinear inverse relation  as the solution to the  problem proposed in the  paper [J. Wang, Nonlinear inverse relations for the Bell polynomials via the Lagrange inversion formula, J. Integer Seq., Vol. 22 (2019), \href{https://cs.uwaterloo.ca/journals/JIS/VOL22/Wang/wang53.pdf}{Article  19.3.8}].  As
  applications of  this  inverse relation, we not only find a short proof  of  another nonlinear inverse relation due to Birmajer et al.,
  but also set up a few convolution  identities concerning  the Mina polynomials.
\end{abstract}
\maketitle
\thispagestyle{plain}
\section{Introduction}
Throughout this paper, we shall adopt the same notation of Henrici \cite{Henrici}.  For instance, we shall use $\mathbb{C}[[t]]$ to denote the ring of formal power series (in short, fps) over  the complex  number field $\mathbb{C}$ and for any $f(t)=\sum_{n\geq 0}a_nt^n\in \mathbb{C}[[t]]$, the coefficient functional
$$\boldsymbol\lbrack t^n\boldsymbol\rbrack f(t)=a_n, n\geq 0.$$
For  convenience,  we define
\begin{align*}
 \mathcal{L}_0=\bigg\{\sum_{n=0}^\infty a_nt^n| a_0\neq 0\bigg\},~~\mathcal{L}_1=\bigg\{\sum_{n=0}^\infty a_nt^n| a_0=0, a_1\neq 0\bigg\}.
\end{align*}
Moreover, for $f(t),g(t)\in \mathbb{C}[[t]]$, $g(t)$ is said to be the composite inverse of $f(t)$ if $f(g(t))=g(f(t))=t$. As conventions, we denote the composite inverse $g(t)$ of $f(t)$ by  $f^{\langle-1\rangle}(t)$.

\begin{yl}\label{inverxinzhi}
 Given $f(t)\in \mathbb{C}[[t]]$,
  $f(t)$ has the  composite inverse  if and only if $f(t)\in \mathcal{L}_1$.
   \end{yl}
   We also need  the ordinary Bell polynomials.
\begin{dy} For integers $n\geq k\geq 0$ and
variables
 $(x_n)_{n\geq  1}$, the sum
 \begin{align}
 \sum_{\sigma_k(n)}\frac{k!}{i_1!i_2!\cdots i_{n-k+1}!}x_{1}^{i_1}x_{2}^{i_2}\cdots x_{n-k+1}^{i_{n-k+1}}\label{bellexpression}
\end{align}
is called  the ordinary Bell polynomial in $x_1,x_2,\ldots,x_{n-k+1}$, where $\sigma_k(n)$ denotes the set of  partitions of $n$ with
 $k$ parts, namely, all nonnegative integers $i_1,i_2,\ldots,i_{n-k+1}$ subject to
\begin{align}
\left\{\begin{array}{l}
i_1+i_2+\cdots+i_{n-k+1}=k\\ i_1+2i_2+\cdots+(n-k+1)i_{n-k+1}=n.
\end{array}\right.
\end{align}
We write $B_{n,k}(x_1,x_2,\ldots,x_{n-k+1})$  for such the Bell polynomial as given by (\ref{bellexpression}).
\end{dy}
As of today, the Bell polynomials have played  very important roles in analysis, combinatorics, and number theory.  It should be pointed out here that
 the above Bell polynomials are in agreement with the exponential Bell polynomials \cite[Definition, p.\ 133]{comtet} with the specialization $x_n\to x_n/n!$ and multiplied by $n!/k!$.

 The ordinary generating function of   the Bell polynomials $B_{n,k}(x_1,x_2,\ldots,x_{n-k+1})$ will be often used in our discussions.
\begin{yl}[{\rm \cite[Lemma 3]{wang}}]\label{final}
For any fps
$f(t)=\sum_{n\geq 1} x_n t^n\in \mathcal{L}_1$,
 it holds
 \begin{align}f^k(t)=\sum_{n=k}^\infty B_{n,k}(x_1,x_2,\ldots,x_{n-k+1})t^n.\label{bellnewone}\end{align}
 \end{yl}

 Aside from the generating function of the Bell polynomials,  it is worthwhile to study  inverse relations lurking behind it. Properly speaking, the term ``inverse"  means a pair of equivalent relations  expressing $(x_n)_{n\geq 1}$ in terms of  the Bell polynomials in variables $(y_n)_{n\geq 1}$  and vice versa. To the best of our knowledge, it is one of the most interesting problems first  posed and solved by Riordan  \cite[Chaps.\ 2 and 3]{riodan},
 and  also investigated  by  Hsu  et al.\ \cite{laoshi} and   Mihoubi  \cite{mihoubi}.
 The reader may consult  Riordan \cite[Sect.\ 5.3]{riodan} for  further details and Mihoubi \cite{mihoubi} for many of such inverse relations.
  It is especially noteworthy  that in their paper \cite{birmajer},  via the establishment of many interesting identities for the Bell polynomials,
   Birmajer  et al.  achieved   the following   somewhat unusual (essentially different from  \cite{laoshi,xrma-huang,riodan})   inverse relation.
\begin{dl}[{\cite[Theorem\ 17]{birmajer}}]\label{Gosper} Let $B_{n,k}(x_1,x_2,\ldots,x_{n-k+1})$ denote the Bell polynomials as above. Then for any integers $a,b,m$ with $m\geq 1$, $a^2+b^2\neq 0$, and
 any sequence $(x_m)_{m\geq 1}$, the system of   nonlinear relations
\begin{align}
z_m(b)=\displaystyle\sum_{k=1}^m\frac{a m+bk}{k(a m+b)}{-a m-b\choose k-1}B_{m,k}(x_1,x_2,\ldots,x_{m-k+1})\label{gosper-idi}
\end{align}
is equivalent to the system of nonlinear  relations
\begin{align}
x_m=\displaystyle\sum_{k=1}^m\frac{1}{k}{a m+bk\choose k-1}B_{m,k}(z_1(b),z_2(b),\ldots,z_{m-k+1}(b)).\label{gosper-idii}
\end{align}
 \end{dl}
In the above theorem and in what follows, we use  the notation
${t\choose n}$ to denote the generalized binomial coefficients $(t)_n/n!$ and
$(t)_n$ to  the usual  falling factorial $t(t-1)\cdots(t-n+1).$

Motivated by  Birmajer et al.'s result,  the first author \cite{wang} established the following nonlinear inverse relation.
\begin{dl}[{\rm \cite[Theorem 5]{wang}}]\label{mainewthm} Let $B_{n,k}(x_1,x_2,\ldots,x_{n-k+1})$ denote the Bell polynomials as above. For any integers $m\geq 1$ and $a,b\in\mathbb{C}$, and
 any sequence $(x_m)_{m\geq 1}$, the system of nonlinear  relations
 \begin{align}
y_m(b)=\displaystyle\frac{1}{am+b}\sum_{k=1}^m{-am-b\choose k}B_{m,k}(x_1,x_2,\ldots,x_{m-k+1})\label{uuu}
\end{align}
is equivalent to the system of nonlinear  relations
\begin{align}
x_m=\displaystyle\frac{1}{am+1}
\sum_{k=1}^m{-(am+1)/b\choose k}b^kB_{m,k}(y_1(b),y_2(b),\ldots,y_{m-k+1}(b)). \label{uuuvvv}
\end{align}
\end{dl}

 Almost at the same time,   Birmajer et al. \cite{birmajer2} found another general and more beautiful nonlinear inverse relation for the Bell polynomials. For future reference, we now reformulate it in the form
\begin{dl}[{\rm \cite[Corollary 2.3]{birmajer2}}]\label{333-old} Let $c=r-q\neq 0$, $pqr\neq 0$. Then
\begin{align}
x_n&=\frac{1}{c}\sum_{k=1}^n\bigg(\frac{q}{q+np}\binom{q+np}{k}
-\frac{r}{r+np}\binom{r+np}{k}\bigg)B_{n,k}(y_1,y_2,\ldots,y_{n-k+1})\label{xnyn}
\end{align}
if and only if
\begin{align}
y_n&=-\sum_{k=1}^n\frac{1}{k!}\prod_{j=1}^{k-1}\big(np+k q+cj-1\big)B_{n,k}(x_1,x_2,\ldots,x_{n-k+1}).
\label{results}\end{align}
\end{dl}

 This result  inspires us to consider the following research problem.

\begin{wenti}\label{wentinew}
For any integers $m\geq 1$, let $p$, $(a_k)_{k=1}^m$, and $(q_k)_{k=1}^m$ be $2m+1$ complex numbers subject to  $$p\neq0,~\sum_{k=1}^m a_k=0,~\sum_{k=1}^m a_kq_k\neq 0.$$  Assume  further that
$F(t)=\sum_{n\geq 1}x_nt^n$ and $\phi(t)=1+\sum_{n\geq 1}y_nt^n$  satisfy
 \begin{align}
F(t/\phi^p(t))=\sum_{k=1}^m a_k\phi^{q_k}(t).\label{equationknown-gen}
\end{align}
Find any relationship between the sequences $(x_n)_{n\geq 1}$ and $(y_n)_{n\geq1}$.
\end{wenti}

In the sequel, the first author \cite{wang} offered the following positive solution to this problem without proof.

\begin{dl}[{\rm \cite[Theorem 14]{wang}}]\label{333-new}
With the same notation and assumptions as above. Then  the system of nonlinear  relations
\begin{align}
x_n=\sum_{k=1}^n \biggl(\sum_{i=1}^m \frac{a_iq_i}{np+q_i}{np+q_i\choose k}\biggr)B_{n,k}(y_1,y_2,\ldots,y_{n-k+1})\label{eqone}
\end{align}
is equivalent to the system of nonlinear  relations
\begin{align}
y_n=\sum_{k=1}^n\frac{\lambda_k(-1/p+n)}{1-pn}B_{n,k}(x_1,x_2,\ldots,x_{n-k+1}),\label{eqtwo-gen}
\end{align}
where $\lambda_n(s)$ are defined recursively by
\begin{align}(n+1)\lambda_{n+1}(s)\sum_{k=1}^m a_kq_k&=-ps\lambda_n(s)-\sum_{k=1}^m a_kq_k\sum_{j=1}^{n}\lambda_{n+1-j}(-q_k/p)j\lambda_{j}(s)
\end{align}
with $\lambda_0(s)=1.$
\end{dl}
 The theme of the present paper is to show Theorem \ref{333-new} in full details. Our main  ingredient is  the classical Lagrange inversion formula, restated as follows.

\begin{yl}[The Lagrange inversion formula: {\rm\cite[p.150, Theorem C, Theorem D]{comtet}}]\label{lag}
 Let $\phi(t)\in \mathcal{L}_0$. Then for any fps $G(t)$, it always holds that \begin{align}
G(t)=\sum_{n=0}^\infty a_n\bigg(\frac{t}{\phi(t)}\bigg)^n,\label{lagformula}
\end{align}
where
\begin{align}
a_n=\frac{1}{n}\boldsymbol\lbrack t^{n-1}\boldsymbol\rbrack G'(t)\phi^n(t).\label{lagformula-coeff}
\end{align}
Hereafter, the prime denotes the formal differentiation with respect to $t$.
\end{yl}

Furthermore, we will investigate  the cases  $m=2$ and $m=3$ of Theorem \ref{333-new}. The former case gives rise to a short proof of  Birmajer et al.'s result, i.e., Theorem \ref{333-old}. The latter  leads us to  a new nonlinear inverse relation as below.
\begin{dl}\label{333-new-new-000}
With the same notation and assumptions as above,  the system of nonlinear  relations
\begin{align}
x_n=\sum_{k=1}^n \biggl(\sum_{i=1}^3 \frac{a_iq_i}{np+q_i}{np+q_i\choose k}\biggr)B_{n,k}(y_1,y_2,\ldots,y_{n-k+1})\label{eqone-new}
\end{align}
is equivalent to the system of nonlinear  relations
\begin{align}
y_n=-\sum_{k=1}^n\frac{f_k(-1/p+n)}{k!(-c_1)^k}B_{n,k}(x_1,x_2,\ldots,x_{n-k+1}),\label{eqtwo-gen-new}
\end{align}
where we define
\begin{align}f_n(t)=\sum_{k=0}^{n-1}
\frac{C_{n,k}(c_1,c_2,\ldots,c_{n-k})}{c_1^{n-1-k}}(pt)^k\label{two-great-new}
\end{align}
where $C_{n,k}(x_1,x_2,\ldots,x_{n-k})$ is the Mina polynomial (see Definition \ref{mina}), and
\begin{align}c_k=a_1q_1^k+a_2q_2^k
+a_3q_3^k.\label{cccoefficient}\end{align}
\end{dl}
Our paper is planned as follows. The next section is devoted to the full proof for Theorem  \ref{333-new}, wherein  the coefficients $\lambda_n(s)$ is introduced and discussed in details. As further applications of Theorem  \ref{333-new} in the cases $m=2$ and $3$,  the proofs of Theorem \ref{333-old} and Theorem \ref{333-new-new-000} are presented in Section 3. Further, some combinatorial identities for $\lambda_n(s)$ are established.
\section{Proofs of the main results}
\subsection{The proof of Theorem \ref{333-new}}
Our proof of Theorem \ref{333-new} is composed of the following lemmas. At first, making use of the Lagrange inversion formula \eqref{lagformula}, it is easy to express $x_n$ in terms of $(y_n)_{n\geq 1}$.
\begin{yl}\label{generalexpress-new} Under the assumptions of {\bf Research Problem~}\ref{wentinew}. We have
\begin{align}
x_n=\sum_{k=1}^n \bigg(\sum_{i=1}^m \frac{a_iq_i}{np+q_i}\binom{np+q_i}{k}\bigg)
B_{n,k}(y_1,y_2,\ldots,y_{n-k+1}).\label{generalexpress-id-new}
\end{align}
\end{yl}
\pf  It suffices to apply  the Lagrange inversion formula \eqref{lagformula} to
   \eqref{equationknown-gen}, namely
\begin{align*}
\sum_{n=1}^\infty x_n(t/\phi^p(t))^n=\sum_{k=1}^m a_k\phi^{q_k}(t).
\end{align*}
Thus we compute directly
\begin{align*}
x_n&=\frac{1}{n}[t^{n-1}]\bigg(\phi^{np}(t)\big(\sum_{i=1}^m a_i\phi^{q_i}(t)\big)^{'}\bigg)\\
&=\frac{1}{n}[t^{n-1}]\bigg(\sum_{i=1}^m a_iq_i\phi^{q_i-1}(t)\bigg)\phi^{'}(t)\phi^{np}(t)\\
&=\sum_{i=1}^m \frac{a_iq_i}{n}[t^{n-1}]\phi^{np+q_i-1}(t)\phi^{'}(t)\\
&=\sum_{i=1}^m \frac{a_iq_i}{n(np+q_i)}[t^{n-1}]\big(\phi^{np+q_i}(t)\big)^{'}\\
&=\sum_{i=1}^m \frac{a_iq_i}{np+q_i}[t^{n}]\phi^{np+q_i}(t)\\
&=\sum_{k=1}^n \bigg(\sum_{i=1}^m \frac{a_iq_i}{np+q_i}\binom{np+q_i}{k}\bigg)B_{n,k}(y_1,y_2,\ldots,y_{n-k+1}).
\end{align*}
Hence \eqref{generalexpress-id-new} is proved.
\qed

All remains to express $y_n$ in terms of $(x_n)_{n\geq 1}$.  For that end, we have to establish  a series of preliminaries. The first one is  a general result about the composite inverse for any fps in $\mathcal{L}_1$.
\begin{yl}\label{importantyl-new} Let $g(t)=t/w(t)$ be the composite inverse of $f(t)=t/\phi^p(t)\in \mathcal{L}_1$. Then it holds
\begin{align}
\phi(t/w(t))&=w^{-1/p}(t).\label{importantphi-new}
\end{align}
\end{yl}
\pf  At first, from
\begin{align*}
f(t)&=t/\phi^p(t)
\end{align*}
  it follows
\begin{align}
\phi(t)=\bigg(\frac{t}{f(t)}\bigg)^{1/p}.\label{important-inver-new}
\end{align}
Replacing $t$ with $g(t)$ in  \eqref{important-inver-new} and recalling $f(g(t))=t$, we further get
\begin{align*}
\phi(g(t))=\bigg(\frac{g(t)}{t}\bigg)^{1/p}.
\end{align*}
It, after $g(t)=t/w(t)$ inserted, turns out be \eqref{importantphi-new}.
\qed

By virtue of Lemma \ref{importantyl-new}, we are able to express $y_n$ in terms of $(x_n)_{n\geq 1}$.

\begin{yl}\label{dl42-new} Under the assumptions of\, {\bf Research Problem~}\ref{wentinew}. We have
\begin{align}y_n=\sum_{k=1}^n\frac{\lambda_k(-1/p+n)}{1-pn}
B_{n,k}(x_1,x_2,\ldots,x_{n-k+1}),
\end{align}
 where $\lambda_n(s)$ are defined by
\begin{align}
w^{s}(t)=1+\sum_{n=1}^\infty \lambda_n(s) F^n(t)\qquad(s\in\mathbb{C}).\label{phidefne-gen-1-new}
\end{align}
\end{yl}
\pf  Observe that \eqref{importantphi-new}  is  equivalent to
\begin{align}
1+\sum_{n=1}^\infty y_n(t/w(t))^n&=w^{-1/p}(t).\label{important-inver-g333-new}
\end{align}
In this form, by the Lagrange inversion formula \eqref{lagformula}, we obtain
\begin{align*}
y_n&=\frac{-1}{pn}[t^{n-1}]w^{-1/p-1}(t)w^{'}(t)w^{n}(t)\nonumber\\
&=\frac{-1}{pn(-1/p+n)}[t^{n-1}](w^{-1/p+n}(t))^{'}=\frac{1}{1-pn}[t^{n}]w^{-1/p+n}(t)\nonumber\\
&=\sum_{k=1}^n\frac{\lambda_k(-1/p+n)}{1-pn}B_{n,k}(x_1,x_2,\ldots,x_{n-k+1}).
\end{align*}
The last identity is based on the definitions \eqref{bellnewone} and  \eqref{phidefne-gen-1-new}. The lemma is proved.
\qed

Actually, Lemma \ref{generalexpress-new}  together with Lemma  \ref{dl42-new}   gives the complete proof
of Theorem \ref{333-new}, except for a full characterization on $\lambda_n(s)$.  It will be discussed   latter (see \eqref{recresults-two}).

Before proceeding, we had better  illustrate Lemma \ref{dl42-new} by the following case.

\begin{lz} Let $\phi(t)=1-t$ and $p=1$. Then from   \eqref{importantphi-new} it follows
$w(t)=1+t.$
Hence we derive from \eqref{important-inver-new}  a combinatorial identity
\begin{align}
\sum_{k=1}^n\lambda_k(n-1)B_{n,k}(x_1,x_2,\ldots,x_{n-k+1})=0 \quad(n\geq 2),\label{forgot}
\end{align}
where $x_n$ and $\lambda_n(s)$ are given respectively by
\begin{align}
x_n&=(-1)^n\sum_{i=1}^m a_i\binom{n-1+q_i}{n}, \\
(1+t)^{s}&=1+\sum_{n=1}^\infty \lambda_n(s)\bigg(\sum_{k=1}^\infty~x_kt^k\bigg)^n.\label{lambdadef}
\end{align}
Furthermore,  it follows from \eqref{lambdadef}  a more general identity than \eqref{forgot}
\begin{align}
\binom{s}{n}=\sum_{k=1}^n \lambda_k(s)B_{n,k}(x_1,x_2,\ldots,x_{n-k+1}).
\end{align}
\end{lz}
This example shows that it seems  difficult to find any closed-form expression for $\lambda_n(s)$, even for $p=1$ or $2$. Thus, it is necessary to  study possible  recurrence relations of the sequence $(\lambda_n(s))_{n\geq 0}$.
\subsection{Recurrence relations for $\lambda_n(s)$}

The following recurrence relations are based on the definition \eqref{phidefne-gen-1-new}.
\begin{yl}\label{recrelation-one} Let $\lambda_n(s)$ be given by \eqref{phidefne-gen-1-new} with $\lambda_0(s)=1$. Then for any $a,b$, there hold
\begin{align}\lambda_{n}(a+b)&=\sum_{k=0}^n \lambda_k(a)\lambda_{n-k}(b),\label{xxx-1}\\
n\lambda_{n}(a+b)&=\frac{a+b}{a}\sum_{k=0}^n k\lambda_k(a)\lambda_{n-k}(b).\label{recresults-one}
\end{align}
\end{yl}
\pf  Evidently, \eqref{xxx-1} is a direct consequence of the basic relation
\[w^{a+b}(t)=w^{a}(t)w^{b}(t).\]
We only need to show \eqref{recresults-one}.
For this end,   recall that
\begin{align}
w^{s}(t)=\sum_{n=0}^\infty \lambda_n(s) F^n(t)\label{id-one}\qquad\mbox{and}\\
sw^{s-1}w^{'}(t)=\sum_{n=0}^\infty n\lambda_n(s) F^{n-1}(t)F^{'}(t).\label{id-two}
\end{align}
Therefore, setting $s=b$ in \eqref{id-one}, we get
\begin{align*}
w^{b}(t)=\sum_{j=0}^\infty \lambda_j(b) F^j(t)
\end{align*}
while putting  $s=a$ and $a+b$ in \eqref{id-two}, respectively,  we easily find
\begin{align*}
aw^{a-1}w^{'}(t)&=\sum_{i=0}^\infty i\lambda_i(a) F^{i-1}(t)F^{'}(t);\\
(a+b)w^{a+b-1}w^{'}(t)&=\sum_{n=0}^\infty n\lambda_n(a+b) F^{n-1}(t)F^{'}(t).
\end{align*}
Upon substituting these expressions into the following basic relation
\begin{align*}
(a+b)w^{a+b-1}w^{'}(t)=\frac{a+b}{a}w^{b}(t)\times aw^{a-1}(t)w^{'}(t),
\end{align*}
we thereby get
\begin{align}
\sum_{n=0}^\infty n\lambda_n(a+b) F^{n}(t)=\frac{a+b}{a}\sum_{j=0}^\infty \lambda_j(b) F^j(t)~ \sum_{i=0}^\infty i\lambda_i(a) F^{i}(t).\label{weigt}
\end{align}
By equating the coefficients of $F^n(t)$ on the both sides of \eqref{weigt}, we finally obtain
\begin{align*}
n\lambda_n(a+b)=\frac{a+b}{a}\sum_{i+j=n} i\lambda_i(a) \lambda_j(b).
\end{align*}
Thus \eqref{recresults-one} is confirmed.
\qed

Now let us return to Research problem \ref{wentinew}. We further establish another recurrence relation by virtue of \eqref{recresults-one}.
\begin{yl}\label{recrelation-two} Let $\lambda_n(s)$ be given by \eqref{phidefne-gen-1-new}, $2m$ parameters $(a_k)_{k=1}^m$ and $(q_k)_{k=1}^m$ satisfy
$$c_0=\sum_{k=1}^m a_k=0, ~~c_1=\sum_{k=1}^m a_kq_k\neq 0.$$
Then
\begin{align}
\lambda_{n}(s)=(n+1)\sum_{k=1}^m\frac{a_kq_k}{q_k-ps}
\lambda_{n+1}(s-q_k/p).\label{recresults-three}
\end{align}
In particular,
\begin{align}(n+1)c_1\lambda_{n+1}(s)=&-ps\lambda_n(s)-\sum_{k=1}^m a_kq_k\sum_{i=1}^{n}\lambda_{n+1-i}(-q_k/p)i\lambda_{i}(s).
\label{recresults-two}
\end{align}
\end{yl}
\pf Observe first that, since
$
c_0=0$ and
$c_1\neq 0$,  $(F^n(t))_{n=0}^\infty$ forms a base for the ring of formal  power series $\mathbb{C}[[t]]$. That means that $\lambda_n(s)$ given by \eqref{phidefne-gen-1-new} are well-defined and unique.  To show \eqref{recresults-three},  we start with  \eqref{equationknown-gen}, namely
\begin{align*}
F(t/\phi^p(t))=\sum_{k=1}^m a_k\phi^{q_k}(t).
\end{align*}
According to Lemma \ref{importantyl-new}, we obtain
\begin{align}
F(t)=\sum_{k=1}^m a_k\phi^{q_k}(t/w(t))=\sum_{k=1}^m a_kw^{-q_k/p}(t).\label{xrmaadded}
\end{align}
Evidently, differentiating  both sides of \eqref{phidefne-gen-1-new} with respect to $t$  leads us to
\begin{align*}
sw^{s-1}(t)w^{'}(t)&=\sum_{n=1}^\infty n\lambda_n(s)F^{n-1}(t)F^{'}(t)\\
&=\frac{1}{p}\sum_{n=1}^\infty n\lambda_n(s)F^{n-1}(t)\bigg(-\sum_{k=1}^m a_kq_kw^{-q_k/p-1}(t)\bigg)w^{'}(t).
\end{align*}
This, after simplified, is equivalent to
\begin{align}
psw^{s}(t)=\sum_{n=1}^\infty n\lambda_n(s)F^{n-1}(t)\bigg(-\sum_{k=1}^m a_kq_kw^{-q_k/p}(t)\bigg).\label{phidefne-gen-1}
\end{align}
Upon taking \eqref{phidefne-gen-1-new} into account, we have
\begin{align*} ps\sum_{n=0}^\infty \lambda_n(s) F^n(t)&=-\sum_{j=0}^\infty (j+1)\lambda_{j+1}(s)F^{j}(t)\bigg(\sum_{k=1}^m a_kq_k\sum_{i=0}^\infty \lambda_i(-q_k/p) F^i(t)\bigg).
\end{align*} By equating the coefficients of $F^n(t)$ on both sides, we get
\begin{align} ps\lambda_n(s) &=-\sum_{k=1}^m a_kq_k\sum_{\stackrel{i+j=n}{i,j\geq 0}}\lambda_{i}(-q_k/p)(j+1)\lambda_{j+1}(s).\label{5tttt}
\end{align}
At this stage, referring to  \eqref{recresults-one}, we are able to evaluate the inner sum on the right side of \eqref{5tttt}. The result is
\begin{align*} &\sum_{\stackrel{i+j=n}{i,j\geq 0}}\lambda_{i}(-q_k/p)(j+1)\lambda_{j+1}(s)=\sum_{\stackrel{i+j=n+1}{i\geq 0,j\geq 1}}\lambda_{i}(-q_k/p)j\lambda_{j}(s)\\
&\qquad=\sum_{\stackrel{i+j=n+1}{i,j\geq 0}}\lambda_{i}(-q_k/p)j\lambda_{j}(s)=\frac{(n+1)ps}{ps-q_k}
\lambda_{n+1}(s-q_k/p).
\end{align*}
Substituting this result into \eqref{5tttt} gives rise to
\begin{align*}\lambda_{n}(s)=(n+1)\sum_{k=1}^m\frac{a_kq_k}{q_k-ps}\lambda_{n+1}(s-q_k/p).
\end{align*}
Thus \eqref{recresults-three} is proved.
We remark that \eqref{recresults-two} is obtainable by solving \eqref{5tttt} for $\lambda_{n+1}(s)$.
The proof is finished.
\qed

\begin{remark}\label{remark21} It is important to observe from   \eqref{recresults-two}  that $\lambda_n(s)$ is polynomial in $ps$ of degree $n$. This fact is useful for our forthcoming discussion.
\end{remark}
\section{Applications}
In this section, we will turn attention to  the proofs of Theorem  \ref{333-old} and Theorem \ref{333-new-new-000}.  Both are the special cases $m=2$ and $3$ of Theorem \ref{333-new}.
\subsection{A short proof of Theorem \ref{333-old}}
\begin{xinzhi}\label{twocase} Let $\lambda_n(s)$ be given by \eqref{phidefne-gen-1-new} with conditions on the parameters that $m=2$, $a_1=-a_2=1/(r-q)$ and
$q_1=q,q_2=r$. Then
\begin{align}
\lambda_n(s)=\frac{p s}{n!} \prod _{k=1}^{n-1} (ps+kq
   +(n-k)r).\label{recrelation-two-1}
\end{align}
\end{xinzhi}
\pf Under  these conditions on the parameters,  we can  combine \eqref{xrmaadded} with \eqref{phidefne-gen-1} used in the proof of
 Lemma \ref{recrelation-two}, thereby obtaining
\begin{align}  psw^s(t) &=\frac{r}{r-q}\sum_{n=1}^\infty n\lambda_n(s)F^{n-1}(t)w^{-r/p}(t)\nonumber\\
&-q\sum_{n=1}^\infty n\lambda_n(s) F^{n-1}(t)\bigg(F(t)+\frac{1}{r-q}w^{-r/p}(t)\bigg)\nonumber\\
&=-q\sum_{n=1}^\infty n\lambda_n(s) F^{n}(t)+\sum_{n=1}^\infty n\lambda_n(s) F^{n-1}(t)w^{-r/p}(t).\label{gggggg}
\end{align}
In this from, by  replacing  $w^{s}(t)$ on the left side  and $w^{-r/p}(t)$ on the right side of \eqref{gggggg} with the expressions given by \eqref{phidefne-gen-1-new} and then equating the coefficients of $F^n(t)$, we have
\begin{align*}  (ps+qn)\lambda_n(s) &=\sum_{\stackrel{i+j=n}{i,j\geq 0}}\lambda_{i}(-r/p)(j+1)\lambda_{j+1}(s)\\
&=\frac{(n+1)s}{s-r/p}\lambda_{n+1}(s-r/p).\end{align*}
Note that the last equality comes from \eqref{recresults-one}. Finally, we have
\begin{align*}\lambda_{n+1}(s-r/p)=\frac{(ps+qn)(ps-r)}{ps(n+1)}\lambda_n(s).
\end{align*}
Next, replace $s$ with $s+r/p$. So it becomes
\begin{align}
\lambda_{n+1}(s)=\frac{ps(ps+qn+r)}{(n+1)(ps+r)}\lambda_n(s+r/p).\label{hhhhh}
\end{align}
By iterating \eqref{hhhhh} $n$ times and noting $\lambda_0(s)=1$,  we thereby show by induction on $n$
\begin{align*}
\lambda_{n+1}(s)&=\bigg(\prod_{j=1}^{n+1}\frac{ps+q(n+1-j)+jr}{j}
\bigg)\frac{ps}{(ps+(n+1)r)}\lambda_{0}(s+(n+1)r/p)\\
&=\frac{ps}{(n+1)!}\prod_{j=1}^{n}(ps+(n+1-j)q+jr).
\end{align*}
Thus \eqref{recrelation-two-1} is confirmed.
\qed

\begin{proof}[Proof of Theorem \ref{333-old}]
 Proposition \ref{twocase} together with Lemmas \ref{generalexpress-new} and \ref{dl42-new} subject to $m=2,a_1=-a_2=1/(r-q)$ and $q_1=q,q_2=r$,  gives  the complete proof of Theorem \ref{333-old}.
\end{proof}
\subsection{A proof of Theorem \ref{333-new-new-000} }
Unlike the case $m=2$, in order to show Theorem \ref{333-new-new-000} corresponding to $m=3$, we need a few preliminaries.
\begin{xinzhi}\label{xinzhi3.2} With the same assumption as in Lemma \ref{recrelation-two}. Let $\lambda_n(s)$ be given by \eqref{phidefne-gen-1-new}  with conditions on the parameters that $m=3$, $a_1+a_2+a_3=0$. Then \begin{align}\frac{ps+nq_3}{(n+1)ps}\lambda_{n}(s)&=\frac{a_1(q_3-q_1)}{ps-q_1}
\lambda_{n+1}(s-q_1/p)+\frac{a_2(q_3-q_2)}{ps-q_2}
\lambda_{n+1}(s-q_2/p),\label{two}\\
\frac{n}{(n+1)ps}\lambda_{n}(s)&=\frac{a_1}{ps-q_1}
\lambda_{n+1}(s-q_1/p)\nonumber\\
&+\frac{a_2}{ps-q_2}
\lambda_{n+1}(s-q_2/p)+\frac{a_3}{ps-q_3}
\lambda_{n+1}(s-q_3/p).\label{three}
\end{align}
\end{xinzhi}
\pf  Under the assumption  \eqref{equationknown-gen}, it holds
\begin{align}
F(t/\phi^p(t))=\sum_{k=1}^3 a_k\phi^{q_k}(t).\label{wangadd}
\end{align}
A direct application of Lemma \ref{importantyl-new} reduces \eqref{wangadd} to
\begin{align}
F(t)=\sum_{k=1}^3 a_k\phi^{q_k}(t/w(t))=\sum_{k=1}^3 a_kw^{-q_k/p}(t).\label{phidefne-gen-1-newnew}
\end{align}
Now,  differentiating both sides of \eqref{phidefne-gen-1-newnew} with respect to $t$, we  obtain
\begin{align*}
sw^{s-1}(t)w^{'}(t)&=\sum_{n=1}^\infty n\lambda_n(s)F^{n-1}(t)F^{'}(t)\\
&=\frac{1}{p}\sum_{n=1}^\infty n\lambda_n(s)F^{n-1}(t)\bigg(-\sum_{k=1}^3 a_kq_kw^{-q_k/p-1}(t)\bigg)w^{'}(t).
\end{align*}
After a bit simplification, it turns out to be
\begin{align}
psw^{s}(t)=\sum_{n=1}^\infty n\lambda_n(s)F^{n-1}(t)\bigg(-\sum_{k=1}^3 a_kq_kw^{-q_k/p}(t)\bigg).\label{wangjin1}
\end{align}
Observe that \eqref{phidefne-gen-1-newnew} implies
\begin{align*}
a_3w^{-q_3/p}(t)=F(t)-a_1w^{-q_1/p}(t)-a_2w^{-q_2/p}(t).
\end{align*}
On substituting this relation into the right side of \eqref{wangjin1},  we obtain
\begin{align*}-psw^{s}(t)&=\sum_{n=1}^\infty n\lambda_n(s)F^{n-1}(t)\bigg(a_1q_1w^{-q_1/p}(t)+a_2q_2w^{-q_2/p}(t)\\
&+q_3\big(F(t)-a_1w^{-q_1/p}(t)-a_2w^{-q_2/p}(t)\big)\bigg).
\end{align*}
By rearrangement of the series on the right side, we arrive  at
\begin{align*}-psw^{s}(t)
&=q_3\sum_{n=1}^\infty n\lambda_n(s)F^{n}(t)\\
&+a_1(q_1-q_3)w^{-q_1/p}(t)\sum_{n=1}^\infty n\lambda_n(s)F^{n-1}(t)\\
&+a_2(q_2-q_3)w^{-q_2/p}(t)\sum_{n=1}^\infty n\lambda_n(s)F^{n-1}(t).
\end{align*}
At this stage, by  replacing each $w^{x}(t)~ (x=s,-q_1/p,-q_2/p)$ on both sides with the expression given by \eqref{phidefne-gen-1-new} and then equating the coefficients of $F^n(t)$, we get
\begin{align}(nq_3+ps)\lambda_n(s)
&=a_1(q_3-q_1)\sum_{i+j=n}(j+1)\lambda_{j+1}(s)\lambda_i(-q_1/p)\nonumber\\
&+a_2(q_3-q_2)\sum_{i+j=n}(j+1)\lambda_{j+1}(s)\lambda_i(-q_2/p).\label{weigt-2}
\end{align}
Recall that  \eqref{recresults-one} asserts
\begin{align*}\sum_{\stackrel{i+j=n}{i,j\geq 0}}(j+1)\lambda_{j+1}(s)\lambda_{i}(-q_k/p)=\frac{(n+1)ps}{ps-q_k}
\lambda_{n+1}(s-q_k/p).
\end{align*}
Using this identity, we are able to evaluate the sums on the right side of \eqref{weigt-2}. The result is
\begin{align}\frac{ps+nq_3}{(n+1)ps}\lambda_{n}(s)=\frac{a_1(q_3-q_1)}{ps-q_1}
\lambda_{n+1}(s-q_1/p)+\frac{a_2(q_3-q_2)}{ps-q_2}
\lambda_{n+1}(s-q_2/p).\label{two-ma}
\end{align}
So \eqref{two} is proved. Next we proceed to prove  \eqref{three}. For this, by referring to \eqref{recresults-three}, we have
\begin{align}\frac{1}{n+1}\lambda_{n}(s)&=\frac{a_1q_1}{q_1-ps}
\lambda_{n+1}(s-q_1/p)\label{dddd}\\
&+\frac{a_2q_2}{q_2-ps}
\lambda_{n+1}(s-q_2/p)+\frac{a_3q_3}{q_3-ps}
\lambda_{n+1}(s-q_3/p).\nonumber
\end{align}
Upon subtracting \eqref{dddd} from \eqref{two-ma}, we  immediately get
\begin{align*}\frac{n}{(n+1)ps}\lambda_{n}(s)&=\frac{a_1}{ps-q_1}
\lambda_{n+1}(s-q_1/p)\\&+\frac{a_2}{ps-q_2}
\lambda_{n+1}(s-q_2/p)+\frac{a_3}{ps-q_3}
\lambda_{n+1}(s-q_3/p).
\end{align*}
Thus, \eqref{three} is proved.
\qed

Next we illustrate the application of  Proposition \ref{xinzhi3.2} as a case study.
\begin{lz} Let  $\lambda_n(s)$ be given by \eqref{phidefne-gen-1-new} and $m=3$, $q_1=q_2$. Then
\begin{align}\lambda_{n}(s)=\frac{ps}{n!(a_3(q_1-q_3))^{n}}
\prod^{n-1}_{k=1}(ps+kq_1+(n-k)q_3).\label{hhhh}
\end{align}
\end{lz}
\pf Observe that the case $q_1=q_2$ reduces \eqref{two} to
\begin{align}\frac{ps+nq_3}{(n+1)ps}\lambda_{n}(s)=\frac{a_3(q_3-q_1)}{q_1-ps}
\lambda_{n+1}(s-q_1/p).\label{three-three}
\end{align}
From \eqref{three-three} it follows
\begin{align*}\lambda_{n+1}(s-q_1/p)
=\frac{ps+nq_3}{(n+1)ps}\frac{q_1-ps}{a_3(q_3-q_1)}\lambda_{n}(s).
\end{align*}
Replace $s$ with $s+q_1/p$. We obtain the
recurrence relation as
\begin{align}
\lambda_{n+1}(s)=\frac{ps+q_1+nq_3}{(n+1)(ps+q_1)}\frac{ps}{a_3(q_1-q_3)}
\lambda_{n}(s+q_1/p).\label{three-three-000}
\end{align}
Iterating \eqref{three-three-000} in succession and then showing by induction,  we arrive at
\begin{align*}
\lambda_{n+1}(s)=\frac{ps}{(n+1)!(a_3(q_1-q_3))^{n+1}}\prod^{n}_{j=1}(ps+jq_1+(n+1-j)q_3).
\end{align*}
The proof is finished.
\qed

Further, in order to  efficiently compute $\lambda_n(s)$, we need to  reformulate the recurrence relation \eqref{recresults-two} as follows.
\begin{xinzhi}\label{prop3.4} Let  $\lambda_n(s)$ be given by \eqref{phidefne-gen-1-new} with $m=3$ and  write
\begin{align}\lambda_ {n}(s)=\frac{ps}{n!(-c_1)^n}f_n(s).\label{kkkk}\end{align}Then $f_0(s)=1/(ps), f_1(s)=1$, and
\begin{align}f_{n+1}(s)=ps f_n(s)+\frac{1}{c_1}\sum_{k=1}^3 a_kq_k^2\sum_{i=1}^{n}\binom{n}{i}f_{i}(-q_k/p)f_{n+1-i}(s).
\label{deflambda}
\end{align}
\end{xinzhi}
\pf Evidently, the case $m=3$ specializes  \eqref{recresults-two} to
\begin{align*}(n+1)c_1\lambda_{n+1}(s)=-ps\lambda_n(s)-\sum_{k=1}^3a_kq_k\sum_{i=1}^{n}\lambda_{i}(-q_k/p)(n+1-i)\lambda_{n+1-i}(s).
\end{align*}
Taking \eqref{kkkk} into account and simplifying the last identity, we obtain \begin{align*}f_{n+1}(s)=ps f_n(s)+\frac{1}{c_1}\sum_{k=1}^3 a_kq_k^2\sum_{i=1}^{n}\binom{n}{i}f_{i}(-q_k/p)f_{n+1-i}(s).
\end{align*}
We have the required result.
\qed

\begin{lz}\label{lzold} Let $c_n$ be given by \eqref{cccoefficient} and  $f_n(s)$ be given by \eqref{deflambda}. Then
\begin{align}f_2(s)&=ps+\frac{c_2}{c_1},\nonumber\\
f_3(s)
&=(ps)^2+3\frac{c_2}{c_1}ps+\frac{3c_2^2-c_1c_3}{c_1^2},\nonumber\\
f_{4}(s)&=(ps)^3+6\frac{c_2}{c_1}(ps)^2+\frac{15c_2^2-4c_1c_3}{c_1^2}ps+
\frac{15c_2^3-10c_1c_2c_3+c_1^2c_4}{c_1^3}.\label{fffff}
\end{align}
\end{lz}
\pf At first, by virtue of the definition \eqref{kkkk}, it is easy to check
\begin{align*}
f_2(s)=ps+\frac{c_2}{c_1}.
\end{align*}
Based on this, it is not hard  to find  by iteration
\begin{align*}
f_3(s)&=ps f_2(s)+\frac{1}{c_1}\sum_{k=1}^3 a_kq_k^2(2f_{1}(-q_k/p)f_{2}(s)+f_{2}(-q_k/p)f_{1}(s))\\
&=ps \frac{ps c_1+c_2}{c_1}+\frac{1}{c_1}\sum_{k=1}^3 a_kq_k^2\bigg(2\frac{ps c_1+c_2}{c_1}+\frac{-q_k c_1+c_2}{c_1}\bigg)\\
&=(ps)^2+3\frac{c_2}{c_1}ps+3\frac{c_2^2}{c_1^2}-\frac{c_3}{c_1}.
\end{align*}
Given $f_2(s)$ and $f_3(s)$, it is easy to  derive from \eqref{deflambda} that
\begin{align*}
f_{4}(s)&=(ps)^3+3\frac{c_2}{c_1}(ps)^2+3\frac{c_2^2}{c_1^2}ps-\frac{c_3}{c_1}ps\\
&+3\frac{c_2}{c_1}f_{3}(s)+\frac{3}{c_1}\sum_{k=1}^3 a_kq_k^2f_{2}(-q_k/p)f_{2}(s)+\frac{1}{c_1}\sum_{k=1}^3 a_kq_k^2f_{3}(-q_k/p).\end{align*}
A direct substitution of $f_2(s)$ and $f_3(s)$ gives rise to
\begin{align*}
f_{4}(s)&=(ps)^3+3\frac{c_2}{c_1}(ps)^2+3\frac{c_2^2}{c_1^2}ps-\frac{c_3}{c_1}ps+3\frac{c_2}{c_1}(ps)^2+9\frac{c_2^2}{c_1^2} ps+9\frac{c_2^3}{c_1^3}-3\frac{c_2c_3}{c_1^2}\\
&+\frac{3}{c_1}\sum_{k=1}^3 a_kq_k^2\bigg(-q_kps+\frac{c_2}{c_1}ps-\frac{c_2}{c_1}q_k+\frac{c_2^2}{c_1^2}\bigg)\\
&+\frac{1}{c_1}\sum_{k=1}^3 a_kq_k^2\bigg(q_k^2-3\frac{c_2}{c_1}q_k+3\frac{c_2^2}{c_1^2}-\frac{c_3}{c_1}\bigg).
\end{align*}
Finally, we obtain
\begin{align*}
f_{4}(s)&=(ps)^3+6\frac{c_2}{c_1}(ps)^2+\bigg(15\frac{c_2^2}{c_1^2}-4\frac{c_3}{c_1}\bigg)ps+15\frac{c_2^3}{c_1^3}-10\frac{c_2c_3}{c_1^2}+\frac{c_4}{c_1}.
\end{align*}
\qed

These computational results about $f_n(s)~(1\leq n\leq 4)$ inspire us to introduce
\begin{dy}[The Mina polynomials]\label{mina} For integers $n\geq 1, m,k,r\geq 0$ with  $n-1\geq k$, we define the $n$ by $n$ (blocked) matrices $A_{n,r}$ by
\begin{align}A_{n,r}:=\begin{pmatrix}
\mathbf{E}_{r}&\mathbf{0}\\\mathbf{0}&M_{n-r}\end{pmatrix},\label{Adef-new}
\end{align}
where $\mathbf{E}_n$  denotes the $n\times n$ identity matrix and  $M_n$ is the $n\times n$ matrix  subject to
\begin{align}
[M_n]_{m+1,k+1}:=\binom{k}{m}(-1)^{k-m}x_{k-m+1}.
\label{Mdef-new}
\end{align}
Then we define the Mina polynomial $C_{n,k}(x_1,x_2,\ldots,x_{n-k})$ to be the entry
\begin{align}x_1^{2n-2-k}[A_{n,0}^{-1}A_{n,1}^{-1}\cdots A_{n,n-2}^{-1}]_{k+1,n}\label{xxxxx-new}
\end{align}
 for $n\geq 2$ and  $C_{1,0}(x_1)=1$ for $n=1,k=0$.
\end{dy}

Hereafter, the notation $A^{-1}$ stands for the usual inverse of the matrix $A=(a_{ij})$ and $[A]_{i,j}$ for its entry $a_{ij}$.

\begin{lz} Let  $C_{n,k}(x_1,x_2,\ldots,x_{n-k})$ be given by \eqref{xxxxx-new}. Then
\begin{align*}
& C_{1,0}(x_1)=1 \\
& C_{2,1}(x_1)=1,C_{2,0}(x_1,x_2)=x_2\\
& C_{3,2}(x_1)=1, C_{3,1}(x_1,x_2)=3x_2, C_{3,0}(x_1,x_2,x_3)=3x_2^2-x_1x_3 \\
& C_{4,3}(x_1)=1, C_{4,2}(x_1,x_2)=6x_2^2, C_{4,1}(x_1,x_2,x_3)=15x_2^2-4x_1x_3,\\
&\qquad\qquad  C_{4,0}(x_1,x_2,x_3,x_4)=15x_2^3-10x_1x_2x_3+x_1^2x_4.
          \end{align*}
\end{lz}

With the help of the Mina polynomials, we can give a characterization on $f_n(s)$.
\begin{xinzhi}\label{recrelation-twotwo} Let $\lambda_n(s)$ and $f_n(s)$ be given by \eqref{phidefne-gen-1-new}  and \eqref{kkkk}, respectively.   Then we have
\begin{align}f_{n}(s)=\sum_{k=0}^{n-1}\frac{
C_{n,k}(c_1,c_2,\ldots,c_{n-k})}{c_1^{n-1-k}}(ps)^k,
\label{two-great}
\end{align}
where $c_k$ is the same as in \eqref{cccoefficient}.
\end{xinzhi}
\pf As indicated in Remark \ref{remark21} and Proposition \ref{prop3.4}, it is now reasonable to assume that for $n\geq 1$
\begin{align}f_{n}(s)=\sum_{k=0}^{n-1}\chi_n(k)(ps)^k.
\label{two-great-newnew}
\end{align}
It only remains to find $\chi_n(k)$. For this,  substitute $\lambda_n(s)$ given by \eqref{kkkk} into \eqref{recresults-three} and make a bit simplification. We have
\begin{align*}
a_1q_1\sum_{k=0}^n\chi_{n+1}(k)(ps-q_1)^k&+a_2q_2\sum_{k=0}^n\chi_{n+1}(k)(ps-q_2)^k\\
&+a_3q_3\sum_{k=0}^n\chi_{n+1}(k)(ps-q_3)^k=c_1ps\sum_{k=0}^{n-1}\chi_n(k)(ps)^k.
\end{align*}
That is
\begin{align*}
\sum_{k=0}^n\chi_{n+1}(k)\sum_{i=0}^k\binom{k}{i}(ps)^i(-1)^{k-i}c_{k-i+1}=c_1\sum_{k=0}^{n}\chi_n(k-1)(ps)^k.
\end{align*}
By equating the coefficients of $(ps)^m$ for $0\leq m\leq n$, we obtain
\begin{align}
\sum_{k=0}^n\chi_{n+1}(k)\binom{k}{m}(-1)^{k-m}c_{k-m+1}=c_1\chi_n(m-1).\label{coeffi-rec-1}
\end{align}
It leads us to a system of linear equations in $n+1$ unknowns
\begin{align}
M_{n+1}\chi_{n+1}=c_1\begin{pmatrix}0\\\chi_n\end{pmatrix}.
\label{coeffi-rec-2}
\end{align}
Here, we set $\chi(k)=0$ for $k<0$ and $M_ {n+1}$ are given by \eqref{Mdef-new}. Moreover, we define the column vector
\begin{align*}
\chi_n:=(\chi_{n}(0),\chi_{n}(1),\ldots,\chi_{n}(n-1))^{T}.\end{align*}
The superscript $T$ indicates the transpose of vectors.
Obviously, $[M_{n+1}]_{m, m}=c_1\neq 0$,   \eqref{coeffi-rec-2} has thus the unique solution
\begin{align}
\chi_{n+1}=c_1M_{n+1}^{-1}\begin{pmatrix}0\\\chi_n\end{pmatrix}.\label{recursive}
\end{align}
In terms of $A_{n,r}$ given by \eqref{Adef-new} and the vector
$
\beta_{n,r}:=\begin{pmatrix}
\mathbf{0}_{r}\\\chi_{n-r}\end{pmatrix},
$
we can reformulate \eqref{recursive} in the form
\begin{align}
\beta_{n+1,0}=c_1A_{n+1,0}^{-1}\beta_{n+1,1}.\label{neewimportant}
\end{align}
Iterating \eqref{neewimportant}  $n$ times yields
\begin{align*}
\beta_{n+1,0}=c_1^{n}A_{n+1,0}^{-1}A_{n+1,1}^{-1}\cdots A_{n+1,n-1}^{-1}\beta_{n+1,n}.
\end{align*}
It is easy to check that $\beta_{n+1,n}^T=(0,0,\ldots,0,1)^T$. Finally, we arrive at
\begin{align}
\beta_{n+1,0}=c_1^{n}[A_{n+1,0}^{-1}A_{n+1,1}^{-1}\cdots A_{n+1,n-1}^{-1}]_{n+1}.\label{neewimportant-added}
\end{align}
Recall that $[A]_n$ denotes the $n$th column of $A$. Hence, by virtue of both \eqref{xxxxx-new} and \eqref{neewimportant-added},
we conclude
\begin{align}
\chi_{n+1}(k)&=[\beta_{n+1,0}]_{k+1}=c_1^{n}[A_{n+1,0}^{-1}A_{n+1,1}^{-1}\cdots A_{n+1,n-1}^{-1}]_{k+1,n+1}\nonumber\\
&=\frac{c_1^{n}}{c_1^{2(n+1)-2-k}}C_{n+1,k}(c_1,c_2,\ldots,c_{n+1-k})\nonumber\\
&=\frac{1}{c_1^{n-k}}C_{n+1,k}(c_1,c_2,\ldots,c_{n+1-k}). \label{xxxxx}
\end{align}
 Hence  \eqref{two-great-newnew}, i.e., \eqref{two-great}  is proved.
\qed

 The following is the detailed computation for the case $n=3$ of \eqref{neewimportant}.
\begin{lz} Let $n=3$.  Then the system of linear equations \eqref{coeffi-rec-2} reduces to
\begin{align*}
\left(
\begin{array}{cccc}
 c_1 & -c_2 & c_3 & -c_4 \\
 0 & c_1 & -2 c_2 & 3 c_3 \\
 0 & 0 & c_1 & -3 c_2 \\
 0 & 0 & 0 & c_1 \\
\end{array}
\right)\begin{pmatrix}\chi_4(0)\\\chi_4(1)\\\chi_4(2)\\\chi_4(3)\end{pmatrix}=c_1\begin{pmatrix}0\\\chi_3(0)\\\chi_3(1)\\
\chi_3(2)\end{pmatrix}.
\end{align*}
Multiplying by the inverse of the coefficient matrix on both sides and iterating the linear relation resulted,   we obtain the
solution as
\begin{align*}
&\begin{pmatrix}\chi_4(0)\\\chi_4(1)\\\chi_4(2)\\\chi_4(3)\end{pmatrix}=c_1\left(
\begin{array}{cccc}
 c_1 & -c_2 & c_3 & -c_4 \\
 0 & c_1 & -2 c_2 & 3 c_3 \\
 0 & 0 & c_1 & -3 c_2 \\
 0 & 0 & 0 & c_1 \\
\end{array}
\right)^{-1}\begin{pmatrix}0\\\chi_3(0)\\\chi_3(1)\\
\chi_3(2)\end{pmatrix}\\
&=c_1^2\left(
\begin{array}{cccc}
 c_1 & -c_2 & c_3 & -c_4 \\
 0 & c_1 & -2 c_2 & 3 c_3 \\
 0 & 0 & c_1 & -3 c_2 \\
 0 & 0 & 0 & c_1 \\
\end{array}
\right)^{-1}\left(
\begin{array}{cccc}
 1 & 0 & 0 & 0 \\
 0 & c_1 & -c_2 & c_3 \\
 0 & 0 & c_1 & -2 c_2 \\
 0 & 0 & 0 & c_1 \\
\end{array}
\right)^{-1}\begin{pmatrix}0\\0\\\chi_2(0)\\\chi_2(1)\end{pmatrix}\\
&=c_1^3\left(
\begin{array}{cccc}
 c_1 & -c_2 & c_3 & -c_4 \\
 0 & c_1 & -2 c_2 & 3 c_3 \\
 0 & 0 & c_1 & -3 c_2 \\
 0 & 0 & 0 & c_1 \\
\end{array}
\right)^{-1}\left(
\begin{array}{cccc}
 1 & 0 & 0 & 0 \\
 0 & c_1 & -c_2 & c_3 \\
 0 & 0 & c_1 & -2 c_2 \\
 0 & 0 & 0 & c_1 \\
\end{array}
\right)^{-1}\left(
\begin{array}{cccc}
 1 & 0 & 0 & 0 \\
 0 & 1 & 0 & 0 \\
 0 & 0 & c_1 & -c_2 \\
 0 & 0 & 0 & c_1 \\
\end{array}
\right)^{-1}
\begin{pmatrix}0\\0\\0\\\chi_1(0)
\end{pmatrix}.
\end{align*}
Note that $\chi_1(0)=1$. We finally obtain
\begin{align*}
\begin{pmatrix}\chi_4(0)\\\chi_4(1)\\\chi_4(2)\\\chi_4(3)\end{pmatrix}
=c_1^3\left(
\begin{array}{cccc}
 \frac{1}{c_1} & \frac{c_2}{c_1^3} & \frac{3 c_2^2-c_1 c_3}{c_1^5} & \frac{15 c_2^3-10 c_1 c_3 c_2+c_1^2 c_4}{c_1^6} \\
 0 & \frac{1}{c_1^2} & \frac{3 c_2}{c_1^4} & \frac{15 c_2^2-4 c_1 c_3}{c_1^5} \\
 0 & 0 & \frac{1}{c_1^3} & \frac{6 c_2}{c_1^4} \\
 0 & 0 & 0 & \frac{1}{c_1^3}
\end{array}
\right)
\begin{pmatrix}0\\0\\0\\ 1
\end{pmatrix}=\begin{pmatrix}\frac{15 c_2^3-10 c_1 c_3 c_2+c_1^2 c_4}{c_1^3}\\\frac{15 c_2^2-4 c_1 c_3}{c_1^2} \\ \frac{6 c_2}{c_1}\\1
\end{pmatrix}.
\end{align*}
It yields
\begin{align*}
f_{4}(s)=(ps)^3+6\frac{c_2}{c_1}(ps)^2+\frac{15c_2^2-4c_1c_3}{c_1^2}ps+
\frac{15c_2^3-10c_1c_2c_3+c_1^2c_4}{c_1^3}.
\end{align*}
It is in agreement with \eqref{fffff} of Example \ref{lzold}. \qed
\end{lz}

We are in a good position to show Theorem \ref{333-new-new-000}.

\begin{proof}[Proof of Theorem \ref{333-new-new-000}] Lemmas \ref{generalexpress-new} and \ref{dl42-new},
 Proposition  \ref{recrelation-twotwo} as well as the definition \eqref{kkkk} together gives  the complete proof of Theorem \ref{333-new-new-000}.
 \end{proof}
\subsection{Convolution identities}
Similar to the Bell polynomials $B_{n,k}(x_1,x_2,\ldots,x_{n-k+1})$, it is  worthwhile to investigate  the Mina polynomials $C_{n,k}(x_1,x_2,\ldots,x_{n-k})$ given by Definition \ref{mina}.
\begin{tl} Let $c_n$ be given by \eqref{cccoefficient}. Then for arbitrary integers $1\leq m\leq n-1$, we have
\begin{align}&C_{n,m}(c_1,c_2,\ldots,c_{n-m})\label{finalll}\\
&=\frac{1}{2^{m+1}-2}\sum_{i+j=m-1}\sum_{k=1}^{n-1}\binom{n}{k}
C_{k,i}(c_1,c_2,\ldots,c_{k-i})C_{n-k,j}(c_1,c_2,\ldots,c_{n-k-j}).\nonumber
\end{align}
In particular,
\begin{align}2C_{n,1}(c_1,c_2,\ldots,c_{n-1})
&=\sum_{k=1}^{n-1}\binom{n}{k}
C_{k,0}(c_1,c_2,\ldots,c_{k})C_{n-k,0}(c_1,c_2,\ldots,c_{n-k}),\label{finalll-new}\\
3C_{n,2}(c_1,c_2,\ldots,c_{n-2})
&=\sum_{k=1}^{n-1}\binom{n}{k}
C_{k,1}(c_1,c_2,\ldots,c_{k-1})C_{n-k,0}(c_1,c_2,\ldots,c_{n-k}).\label{finalll-new-new}
\end{align}
\end{tl}
\pf It is clear from \eqref{xxxxx} that for $0\leq k\leq n-1$, it holds
\begin{align}\chi_{n}(k)=\frac{C_{n,k}(c_1,c_2,\ldots,c_{n-k})}{c_1^{n-1-k}}.\label{xxx-3}
\end{align}
To establish \eqref{finalll}, we only need to make the replacement $(a,b)\to(as,bs)$ in   \eqref{xxx-1} of  Lemma \ref{recrelation-one}. As a result,  we achieve
\begin{align}
\lambda_{n}(as+bs)=\lambda_n(as)+\lambda_{n}(bs)+\sum_{k=1}^{n-1} \lambda_k(as)\lambda_{n-k}(bs).\label{xxx-2}
\end{align}
Next, in view of \eqref{kkkk}, we may reformulate \eqref{xxx-2}  in terms of $f_n(s)$ and then substitute \eqref{two-great-new} into, obtaining
\begin{align*}\frac{p(a+b)s}{n!c_1^{n}}&
\sum_{k=0}^{n-1}\chi_n(k)(p(a+b)s)^k-\frac{pas}{n!c_1^{n}}
\sum_{k=0}^{n-1}\chi_n(k)(pas)^k-\frac{pbs}{n!c_1^{n}}
\sum_{k=0}^{n-1}\chi_n(k)(pbs)^k\\
&=\sum_{k=1}^{n-1}\bigg(\frac{pas}{k!c_1^{k}}
\sum_{i=0}^{k-1}\chi_k(i)(pas)^i\bigg) ~\bigg(\frac{pbs}{(n-k)!c_1^{n-k}}
\sum_{j=0}^{n-k-1}\chi_{n-k}(j)(pbs)^j\bigg).
\end{align*}
A bit simplification reduces it to
\begin{align*}(a+b)&
\sum_{k=0}^{n-1}\chi_n(k)(p(a+b)s)^k-a
\sum_{k=0}^{n-1}\chi_n(k)(pas)^k-b
\sum_{k=0}^{n-1}\chi_n(k)(pbs)^k\\
&=abps\sum_{k=1}^{n-1}\binom{n}{k}\bigg(\sum_{i=0}^{k-1}\chi_k(i)(pas)^i\bigg) ~\bigg(\sum_{j=0}^{n-k-1}\chi_{n-k}(j)(pbs)^j\bigg).
\end{align*}
By equating the coefficients of $(ps)^m (0\leq m\leq n-1)$ on both sides, we have
\begin{align}&\quad\big((a+b)^{m+1}-a^{m+1}-b^{m+1}\big)\chi_n(m)\nonumber\\&=ab\sum_{k=1}^{n-1}\binom{n}{k}\sum_{i+j=m-1}\chi_k(i)a^ib^j
\chi_{n-k}(j)\nonumber\\
&=ab\sum_{i+j=m-1}a^ib^j\sum_{k=1}^{n-1}\binom{n}{k}\chi_k(i)\chi_{n-k}(j).\label{cccc}
\end{align}
Since the parameters  $a,b$ are arbitrary and $\chi_n(k)$ are independent of $a,b$, we are able to set $a=b$ in \eqref{cccc} and then equate the coefficients of $a^{m+1}$, obtaining
\begin{align*}2(2^{m}-1)\chi_n(m)=\sum_{i+j=m-1}\sum_{k=1}^{n-1}\binom{n}{k}\chi_k(i)\chi_{n-k}(j).
\end{align*} After \eqref{xxx-3} inserted,  it becomes
\begin{align*}&\frac{C_{n,m}(c_1,c_2,\ldots,c_{n-m})}{c_1^{n-1-m}}\\
&=\frac{1}{2^{m+1}-2}\sum_{i+j=m-1}
\sum_{k=1}^{n-1}\binom{n}{k}
\frac{C_{k,i}(c_1,c_2,\ldots,c_{k-i})}
{c_1^{k-1-i}}~\frac{C_{n-k,j}(c_1,c_2,\ldots,c_{n-k-j})}{c_1^{n-k-1-j}}.
\end{align*}
A direct simplification gives  \eqref{finalll}. Evidently, the cases  $m=1$ and $m=2$ specialize \eqref{finalll} to \eqref{finalll-new} and \eqref{finalll-new-new}, respectively.
\qed

In general, we have
\begin{tl}Let $c_n$ be given by \eqref{cccoefficient}. Then for arbitrary integers $1\leq m\leq n-1$, we have
\begin{align}&\qquad(m+1)C_{n,m}(c_1,c_2,\ldots,c_{n-m})\label{finalll-new-new-new}\\
&=\sum_{k=m}^{n-1}\binom{n}{k}
C_{k,m-1}(c_1,c_2,\ldots,c_{k-m+1})C_{n-k,0}(c_1,c_2,\ldots,c_{n-k}).\nonumber
\end{align}
\end{tl}
\pf It suffices to compare the coefficients of $a^m$ on both sides of \eqref{cccc}. Then we have
 \begin{align*}\binom{m+1}{m}\chi_n(m)=\sum_{k=1}^{n-1}\binom{n}{k}\chi_k(m-1)\chi_{n-k}(0).
\end{align*}
It is,
 after \eqref{xxxxx} inserted,
 equivalent to \eqref{finalll-new-new-new}
\qed
\section{Acknowledgments}
 The first author is supported  by NSF of Zhejiang Province (Grant~No.~LQ20A010004) and NSF of China (Grant~No.~12001492) and the second author is supported  by  NSF of China (Grant~No.~11971341).

\end{document}